\documentclass[reqno,12pt]{amsart}

\usepackage{amsmath,amsthm,amssymb,epsfig}
\usepackage{tikz}
\usepackage{float}
\usepackage{mathtools}
\usepackage{subcaption}
\usepackage{algorithm}
\usepackage{algorithmic}
\usepackage{MnSymbol}
\usepackage{graphicx,xcolor}
\usepackage{comment}
\usepackage{tkz-graph}
\usetikzlibrary{fit}
\usetikzlibrary{shapes}
\usepackage{comment}
\usepackage{enumitem}
\usepackage[left=3.5cm,top=3.5cm,right=3.5cm,bottom=3.5cm]{geometry}

\newtheorem{theorem}{Theorem}

\title{Locally finite graphs and their localization numbers}

\author[A.\ Bonato]{Anthony Bonato}
\author[F.\ Lehner]{Florian Lehner}
\author[T.G.\ Marbach]{Trent G.\ Marbach}
\author[JD Nir]{JD Nir}
\address[A1,A3]{Toronto Metropolitan University, Toronto, Canada}
\address[A2]{University of Auckland, Auckland, New Zealand}
\address[A4]{Oakland University, Rochester, U.S.A.}

\email[A1]{(A1) abonato@torontomu.ca}
\email[A2]{(A2) florian.lehner@auckland.ac.nz}
\email[A3]{(A3) trent.marbach@torontomu.ca}
\email[A4]{(A4) jdnir@oakland.edu}

\begin{document}

\keywords{localization number, pursuit-evasion games, locally finite graphs}
\subjclass{05C57,05C63,05C12}

\maketitle

\begin{abstract}
We study the Localization game on locally finite graphs trees, where each of the countably many vertices have finite degree. In contrast to the finite case, we construct a locally finite tree with localization number $n$ for any choice of positive integer $n$. Our examples have uncountably many ends, and we show that this is necessary by proving that locally finite trees with finitely or countably many ends have localization number at most 2. Finally, as is the case for finite graphs, we prove that any locally finite graph contains a subdivision where one cop can capture the robber. 
\end{abstract}

\section{Introduction}

Pursuit-evasion games are most commonly studied on finite graphs, but various studies have also considered the infinite case, such as \cite{bht,hahn,ILW22,fL16}; see also Chapter~7 of \cite{BN11}. The difference between pursuit-evasion games such as the Localization game or Cops and Robbers on finite and infinite graphs is that in the infinite case, the evader may avoid capture by moving along an infinite path without ever visiting the same vertex twice.  In this paper, we present the first extensive study of the Localization game on \emph{locally finite} graphs, that is, infinite graphs in which every vertex has finite degree.\footnote{The present paper is the full version of the extended abstract \cite{BLNM} from EUROCOMB'23.} 

All graphs considered in this paper are simple, connected, and locally finite. The reader is directed to \cite{aB22,rD17} for additional background on graph theory and infinite graphs.

The Localization game was first introduced for one cop independently by Carraher et al.\ \cite{CCDEW12} and by Seager~\cite{sS12, sS14}, and was subsequently studied in several papers such as~\cite{BBHMP22, BHM21, BK20, BGGNS18a, BGGNS18b, BDELM17}. It is a game played between two players playing on a graph. One player controls a set of $k$ \emph{cops}, and the other controls a single \emph{robber}. The players play over a sequence of discrete time-steps; a \emph{round} of the game is a move by the cops and the subsequent move by the robber. The players move on alternate time-steps, with the robber going first. In the first round, the robber occupies a vertex of the graph, and in each subsequent round, they may move to a neighboring vertex or remain on their current vertex. Each move of the cops consists of occupying a set of vertices $u_1, u_2, \ldots, u_k$, and sending out a \emph{cop probe} from each of these vertices. Each cop probe returns the distance $d_i$ from $u_i$ to the robber. We refer to $D = (d_1, d_2, \ldots , d_k)$ as the distance vector of cop probes. Note that the cops are not limited to moving to neighboring vertices. Relative to the cops’ position, there may be more than one vertex $x$ with the same distance vector. We refer to such a vertex $x$ as a \emph{candidate} of $D$ or simply a \emph{candidate}.

The cops win if they have a strategy to determine, after a finite number of rounds, a unique candidate, at which time we say that the cops {\em capture} the robber. The robber wins by evading capture indefinitely. We assume the robber is \emph{omniscient}, in that they know the entire strategy for the cops. For a graph $G$, the \emph{localization number} of $G$, written $\zeta(G)$, is the smallest cardinal for which $k$ cops have a winning strategy.  For further background on the localization number of finite graphs, see Chapter~5 of \cite{aB22}.

We present new results on the localization number of locally finite graphs and trees, paying particular attention to whether results persist or change from the finite case. As connected, locally finite graphs are countable, $\zeta(G)$ is either a positive integer or the first infinite cardinal, $\aleph_0.$ 

For finite trees, it is known that the localization number is at most 2; see \cite{sS14}. In contrast, in Section~2, we construct a locally finite tree with localization number $n$ for any choice of $n$, where $n$ is a positive integer or $\aleph_0.$ These constructions have uncountably many ends; the precise definition of ends will be given later, but for trees, they can be thought of as ``infinite branches'' of the tree. We also prove that trees with countably many ends have localization number of at most 2, giving a broad generalization of the same bound for finite trees. 

If $uv$ is an edge in a graph $G$, then \emph{subdividing} $uv$ means deleting the edge $uv$ and adding a path between $u$ and $v$ whose internal vertices are disjoint from $V(G)$. We say that an edge is subdivided $k$ times if this path has $k$ internal vertices. A graph derived from $G$ by a sequence of subdivisions is called a \emph{subdivision} of $G.$
In Section~3, we consider more general locally finite graphs. For finite graphs, it was shown in \cite{CCDEW12} that some subdivision has localization number 1. We extend this result to the locally finite case and prove that each locally finite graph has a subdivision with localization number 1. 

\section{Trees}

Although determining the localization number for general graphs is \textbf{NP}-hard~\cite{BGGNS18a}, the following theorem of Seager fully characterizes the localization number of finite trees. Let $\hat{T}$ be the tree depicted in Figure~\ref{ft3}.

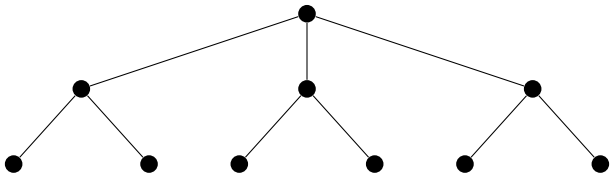
\begin{figure}[ht!]
\centering
\begin{tikzpicture}
\node[fill,circle, scale=0.6] (r) at (0,0) {};
\node[fill,circle, scale=0.6] (a) at (-3,-1) {};
\node[fill,circle, scale=0.6] (b) at (0,-1) {};
\node[fill,circle, scale=0.6] (c) at (3,-1) {};
\node[fill,circle, scale=0.6] (a1) at (-3.9,-2) {};
\node[fill,circle, scale=0.6] (a2) at (-2.1,-2) {};
\node[fill,circle, scale=0.6] (b1) at (-0.9,-2) {};
\node[fill,circle, scale=0.6] (b2) at (0.9,-2) {};
\node[fill,circle, scale=0.6] (c1) at (2.1,-2) {};
\node[fill,circle, scale=0.6] (c2) at (3.9,-2) {};

\draw (a1) -- (a) -- (a2);
\draw (b1) -- (b) -- (b2);
\draw (c1) -- (c) -- (c2);
\draw (a) -- (r) -- (b);
\draw (c) -- (r);
\end{tikzpicture}
\caption{The tree $\hat{T}$.}\label{ft3}
\end{figure}

A graph is $H$-\emph{free} if it does not contain $H$ as an induced subgraph.
\begin{theorem}[\cite{sS14}] \label{thm:seager}
If $T$ is a finite tree, then $\zeta(T) = 1$ if $T$ is $\hat{T}$-free, and $\zeta(T)=2$ otherwise.
\end{theorem}

The Localization game on locally finite trees, however, has received far less attention. While proving a result for finite graphs, Haselgrave, Johnson, and Koch gave the first theorem extending the Localization game to an infinite tree.
\begin{theorem}[\cite{HJK18}] \label{thm:delta_lower_bound}
The infinite $\Delta$-regular tree $T_\Delta$ satisfies $\zeta(T_\Delta) \ge \lfloor \frac{\Delta^2}{4} \rfloor$.
\end{theorem}

As a consequence of Theorems~\ref{thm:seager} and \ref{thm:delta_lower_bound}, we note that locally finite trees offer a richer spectrum of localization numbers than finite trees. We show that for any choice of $n$, including $\aleph_0$, there is a locally finite tree with localization number $n$.

\begin{theorem}\label{thm:achievable}
If $n$ is a positive integer or $n=\aleph_0$, then there is a locally finite tree $T$ with $\zeta(T) = n$.
\end{theorem}

\begin{proof}
For $n =1$, a one-way infinite path (also called a ray) is an example. By probing the unique vertex of degree one, the cop captures the robber in the first round.

We use Theorem~\ref{thm:delta_lower_bound} to construct an example that requires countably many cops. Let $T_\omega$ be the rooted tree where each vertex at distance $i$ from the root has $i+1$ children. We note that $T_\omega$ is locally finite and connected, hence it has countably many vertices, and thus $\zeta(T_\omega)\leq \aleph_0$. On the other hand, for any number of cops $m < \aleph_0$, the robber can choose to play entirely within a subtree of $T_\omega$ isomorphic to $T_{2\lceil \sqrt{m+1} \rceil}$. Even if the cops knew the robber was implementing this strategy, they would require at least
\[ \left \lfloor \frac{(2\lceil \sqrt{m+1} \rceil)^2}{4} \right \rfloor \ge m+1 > m \]
cops to capture the robber.

Finally, we describe a locally finite tree requiring exactly $n$ cops for each $2 \le n < \aleph_0$. To do so, let $T(n)$ be the infinite rooted tree in which every vertex has $n(n-1)+1$ children and subdivide each edge $n-1$ times. We refer to the resulting graph as $S(n)$. The original vertices of $T(n)$ are referred to as \emph{branch nodes}, and the vertices added while subdividing the edges are referred to as \emph{path nodes}. We say that a vertex $v$ of $S(n)$ \emph{descends} from a vertex $u$ if $u$ lies on the unique path from the root to $v$.

We first demonstrate that $n$ cops can capture the robber on $S(n)$. Roughly, their strategy is as follows: at each branch node, the cops determine which subtree contains the robber and move to the next branch node on that subtree. We will show that the cops can do this faster than the robber can flee further down the tree, implying that they eventually probe close enough to the robber to capture them.

Over $n-1$ rounds, the $n$ cops probe $n(n-1)$ of the $n(n-1)+1$ branch vertices closest to the root. If the robber moves to the root during any of these rounds, then each probe will return distance exactly $n$, which uniquely identifies the root, and the robber is captured. Therefore, in order to avoid capture, the robber must not leave their current subtree. We claim that over these $n-1$ rounds the cops either capture the robber or determine on which subtree they reside. The cops can then iterate this strategy, treating the first branch vertex on the subtree containing the robber as the new root. During these $n-1$ rounds, the robber moves at most $n-1$ vertices further from the root, but the cops move $n$ vertices closer to the robber, so each iteration moves the cops at least one vertex closer to the robber. Since the robber started at finite distance from the root, the cops eventually overtake and capture the robber.

We now demonstrate how the cops determine which tree contains the robber. On any round in which each probe returns the same distance, the cops know the robber is not located on any of the probed branches (unless, as previously mentioned, the robber is captured at the root). It is either the case that in each of the $n-1$ rounds, each probe returns the same distance to the robber, or else there is some round in which one probe is closer to the robber than the rest. In the first case, the cops know the robber is located on the unique unprobed subtree. If the distance returned by the probes in the final round is less than $2n$, then the robber is captured on the path between the root and the first branch vertex on the unprobed branch. Otherwise, the cops have determined that the robber is in the subtree rooted at that first branch vertex. In the second case, if one cop receives a distance of less than $n$ in a round when the other cops receive identical distances smaller than $2n$, then the robber is captured on the path between the closest cop and the root. Otherwise, the robber is located in the subtree rooted at the vertex closest to the robber. In either case, the cops have identified the next branch vertex in their strategy.

Now we prove that if the cop player uses only $n-1$ cops, the robber can evade capture. 
In $n$ rounds, the $n-1$ cops may probe a vertex in at most all but one of the connected components of $S(n)$ with the root removed. 
Let $u$ be a branch node of distance $n$ from the root whose connected component is not probed by a cop in the first $n$ rounds. 
Suppose that during the first round, the robber chose to be on the unique vertex adjacent to the root with a distance $n-1$ from $u$. Over the first $n-1$ rounds, there are multiple unprobed components, so the cops cannot distinguish the robber's position from equivalent positions on those other components. The robber then moves away from the root towards the next branch vertex, say $s$, reaching it on the $(n-1)$st round. On round $n$ the robber leaves $s$, choosing a component that, when removing $s$, will not be probed during the next $n$ rounds. Even if the cops are then able to determine which component the robber originally chose, they can now not distinguish between the children of $s$, so the robber continues to avoid capture. The game is now identical to the starting position: the cops have identified a subtree containing the robber, the robber is on a vertex adjacent to the root of that subtree, and the robber has chosen a component that will not be probed in the next $n$ rounds. Thus, by repeating this strategy, the robber avoids capture indefinitely.
\end{proof} 

Note that in the trees constructed in the proof of Theorem \ref{thm:achievable}, the robber evades capture by moving along an infinite path, never visiting the same vertex twice. The reason the cops are not able to pin down the location of the robber is that there are many rounds in which the robber has several branches along which they could continue moving away. It will turn out that having these branches is necessary for a tree to have large localization number.

We use the notion of ends to formalize this statement. A \emph{ray} is an infinite one-way path, and a \emph{double ray} is an infinite two-way path. Two rays are called \emph{equivalent} if there is a third ray sharing infinitely many vertices with each of them. An \emph{end} is an equivalence class of rays with respect to this equivalence relation. We note that if two rays lie in different ends, then they can be separated by removing finitely many vertices, matching our concept of separate infinite branches. 

The theory of ends in general locally finite graphs is expansive (see \cite{rD17}); in the context of locally finite trees, however, things are more straightforward. In particular, two rays of a tree are equivalent if and only if they eventually coincide, or more precisely, if they share a common subray. Moreover, for a given vertex $v$ and end $\epsilon$ of a tree, there is a unique ray starting at $v$, which belongs to $\epsilon$. It is also known that a tree has uncountably-many ends if and only if it contains the infinite binary tree as a minor, see~\cite{rD17}.

Note that each example with localization number more than $2$ given in Theorem \ref{thm:achievable} contains the infinite binary tree and thus has uncountably-many ends. We will show that this is no coincidence: all locally finite trees with at most countably many ends have localization number at most $2$. We start by considering trees with finitely many ends.

\begin{theorem}\label{lem:finite_ends}
If $T$ is a locally finite tree with finitely many ends, then $\zeta(T) \le 2$.    
\end{theorem}

\begin{proof}
We prove a stronger claim: if $T$ is a locally finite tree with finitely many ends and $v\in V(T)$, then two cops can capture the robber such that if the robber ever moves to $v$, then they are captured immediately after the next cop move. We proceed by induction on the number of ends of $T$.

If $T$ contains no ends, then it contains no rays, so as $T$ is locally finite, it must be finite. 
The following strategy, described in~\cite{aB22}, leads to the capture of the robber without letting them occupy $v$: while one cop repeatedly probes $v$ (thereby catching the robber if they attempt to occupy $v$), the second cop sequentially probes the neighbors of $v$.
The cop placed on the neighbors of $v$ will probe a distance that is smaller than the cop probing $v$ exactly once, which is when it probes the subtree containing the robber. 
The cops then switch roles, with the cop who previously probed $v$ sequentially probing the neighbors of the vertex adjacent to $v$ closest to the robber. 
By repeating this process, the cops push the robber towards a leaf, where they are captured before they can return to $v$.

Next, assume that $T$ has at least one end, and let $R = (v = v_0,v_1,v_2,\dots)$ be a ray starting at $v$. Let $S$ be the graph obtained from $T$ by removing all edges of $R$, and let $S_i$ be the connected component of $S$ containing $v_i$. Note that every $S_i$ has fewer ends than $T$ because each ray in $S_i$ belongs to some end of $T$, but none of them belongs to the same end as $R$.

In the first round, both cops probe $v$. Let $d$ be the distance returned by this probe. If $d= 0$, the robber is captured. Otherwise, the first cop keeps probing $v$ receiving distance $d_1$, and the second cop probes $v_{i}$ for $i = d+1,d,d-1,d-2,\dots$ receiving distance $d_2$. In each of these rounds, $d_1+d_2 = i$ if and only if the robber's location is between $v$ and $v_{i}$ on $R$; in this case, the robber is located at $v_{d_1}$ and therefore is captured. Further, $d_1-d_2 = i$ if and only if the robber's location is in $S_i$, and  $d_2-d_1 = i$ if and only if the robber's location is in $S_0$. In this case, since $S_i$ has fewer ends than $T$, the cops can capture the robber within finitely many steps by the induction hypothesis. Since the robber started out in some $S_i$ for $i \leq d$ and would be captured if they return to $v$, one of these two options must eventually happen. Therefore, the robber is captured after finitely many steps.
\end{proof}

As mentioned, we may relax the condition of finitely many ends in the statement of Theorem~\ref{lem:finite_ends} even further. Perhaps surprisingly, two cops also have a winning strategy in case there are countably many ends.

\begin{theorem}\label{thm:countable_ends}
If $T$ is a locally finite tree with countably many ends, then $\zeta(T) \le 2$.    
\end{theorem}

For the proof of Theorem~\ref{thm:countable_ends}, we use transfinite induction on a certain ordinal labeling of ends. The base case for the induction uses Theorem~\ref{lem:finite_ends}. Before stating the proof we need a few definitions.

Given a locally finite rooted tree $T$ we define a partial order, called the \emph{tree order} on $V(T)$ by $u \sqsubseteq v$ if $u$ is on the unique path from $v$ to the root. Let Ord by the class of ordinals. A \emph{recursive pruning} is a labeling $\ell\colon V \to \mathrm{Ord}$ of the vertices of $T$ by ordinal numbers constructed by transfinite recursion as follows. We point out that the recursive definition of the labels ensures that if $v$ is labeled, then so is every $w$ such that $v \sqsubseteq w$. Moreover, $v \sqsubseteq w$ implies that $\ell(v) \geq \ell(w)$.

We start the construction by setting $\ell(v) = 0$ if the elements of the set $\{u \in V(T) \mid v \sqsubseteq u\}$ are pairwise comparable with respect to the tree order. Note that this implies that $\ell(w) = 0$ for every $w$ such that $v \sqsubseteq w$.

For the recursive step, let $\alpha$ be the smallest ordinal which does not appear as a vertex label. Let $T_\alpha$ be the tree obtained from $T$ by removing all vertices $w$ satisfying $\ell(w) < \alpha$. Note that $T_\alpha$ is connected for all $\alpha,$ and the tree order on $T_\alpha$ coincides with that on $T$.
Label a vertex $v \in V(T_\alpha)$ with label $\ell(v) = \alpha$ if any two vertices in the set $\{u \in V(T_\alpha) \mid v \sqsubseteq u\}$ are comparable with respect to the tree order.
In other words, after pruning all vertices labeled so far, assign label $\alpha$ to all vertices after the point where any path or ray starting at the root stops branching. Note that every $w\in V(T_\alpha)$ with $v \sqsubseteq w$ receives label $\ell(w) = \alpha$.

We refer the reader to~\cite[Chapter 8]{rD17} for more background on recursive prunings. The only fact we will need is that a tree $T$ has a recursive pruning (that is, every vertex receives a label in the above procedure) if and only if $T$ does not contain a subdivision of the infinite binary tree, see~\cite[Proposition 8.5.1]{rD17}. Since the infinite binary tree has uncountably many ends, it follows that every rooted tree with countably many ends has a recursive pruning. 

\begin{proof}[Proof of Theorem~\ref{thm:countable_ends}]
Let $T$ be a tree with at most countably many ends. Let $\ell$ be the labeling coming from a recursive pruning with root vertex $r$. As usual for rooted trees, we define the parent of a vertex $v$ to be the second vertex of the unique path from $v$ to $r$, and we say that $v$ is a child of $w$ if $w$ is the parent of $v$. 

Recall that $v \sqsubseteq w$ implies that $\ell(v) \geq \ell(w)$. In particular, if $\ell(v) = \alpha$, then $\ell(w) \leq \alpha$ for all children of $v$.
 
Next, note that the labels along each ray $R$ starting at $r$ are weakly decreasing.  Since decreasing sets of ordinals are finite, there is a sub-ray of $R$ all of whose vertices have the same label. We call this label the \emph{end label} of the end $\epsilon$ containing $R$ and denote it by $\ell(\epsilon)$. Note that $\ell(\epsilon) \geq \alpha$ if and only if $R$ is a ray of $T_\alpha$.

We claim that the supremum of the end labels is, in fact, a maximum and that there are only finitely many ends attaining this maximum. To prove this claim, let us call a vertex of $T$ \emph{essential} if it lies on a ray starting at the root. For each $i$, let $\alpha_i$ be the maximal label of an essential vertex at distance $i$ from the root. Note that this exists because the set of essential vertices at each given distance is finite. Since all vertices on the path from an essential vertex to the root are also essential, we know that the sequence $\alpha_{i}$ is decreasing. There is no infinite, strictly decreasing sequence of ordinal numbers, so there must be some $j \geq 0$ such that $\alpha_i = \alpha _j$ for every $i \geq j$. Define $\alpha= \alpha_j$.

Note that $T_\alpha$ is infinite because there is an essential vertex with label $\geq \alpha$ at each possible distance from $r$. In particular, there is an end with end label $\alpha$. On the other hand, every ray starting at the root must pass through some essential vertex at distance $j$ from the root. Since the labels along such a ray are decreasing, we conclude that $\ell(\epsilon) \leq \alpha$ for every end $\epsilon$, hence $\alpha$ is the maximal end label.

Let $v$ be a vertex at distance $j$ from the root, and assume that there is a ray in $T_\alpha$ starting from the root which contains $v$. We then have that $v$ is essential, and therefore $\ell (v) \leq \alpha$. Note that  $\ell (v) \geq \alpha$ since $v \in T_\alpha$ and thus, $\ell(v) = \alpha$. This implies that any two elements in the set  $\{u \in V(T_\alpha) \mid v \sqsubseteq u\}$ are comparable, and hence, this set forms a ray in $T_\alpha$. Therefore, there is exactly one ray in $T_\alpha$ starting from the root and passing through $v$. It follows that the number of ends of $T_\alpha$ is at most the number of vertices at distance $j$ from the root, and therefore this number is finite.

We prove that two cops have a winning strategy by transfinite induction on the largest end label in the recursive pruning of $T$. For the base case, if $T$ contains no ends, or if the largest end label is $0$, then given that there are finitely many ends with end label $0$, the result follows from Theorem~\ref{lem:finite_ends}. 

Assume now that the theorem holds if the largest end label is strictly less than $\alpha$ and let $T$ be a tree with the largest end label $\alpha$. By implementing a strategy similar to that used to prove Theorem~\ref{lem:finite_ends}, two cops repeatedly restrict the robber's access to ends with label $\alpha$ until the robber is trapped on a subgraph with ends which have label less than $\alpha$. At this point, the cops have a winning strategy by the induction hypothesis, and the theorem follows.
\end{proof}

Given Theorem~\ref{thm:countable_ends}, it is natural to ask if there is a version of Theorem~\ref{thm:seager} for locally finite trees with countably many ends. However, unlike in the finite case, $T_3$ is not the only obstruction to a locally finite tree having localization number one. One such example is the \emph{doubly infinite comb} graph $T_1^\infty$ consisting of a double ray with a leaf attached to each vertex; see Figure~\ref{comb}. 

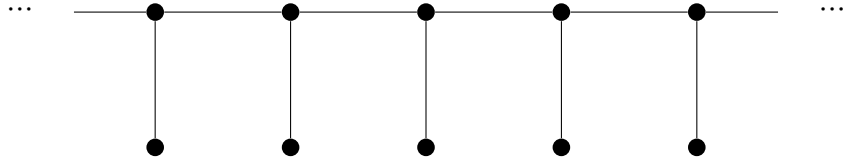
\begin{figure}[ht!]
\centering
\begin{tikzpicture}[scale=1.8]
\node at (-3,0) {$\cdots$};
\node[fill, circle, scale=0.6] (a) at (-2,0) {};
\node[fill, circle, scale=0.6] (a1) at (-2,-1) {};
\node[fill, circle, scale=0.6] (b) at (-1,0) {};
\node[fill, circle, scale=0.6] (b1) at (-1,-1) {};
\node[fill, circle, scale=0.6] (c) at (0,0) {};
\node[fill, circle, scale=0.6] (c1) at (0,-1) {};
\node[fill, circle, scale=0.6] (d) at (1,0) {};
\node[fill, circle, scale=0.6] (d1) at (1,-1) {};
\node[fill, circle, scale=0.6] (e) at (2,0) {};
\node[fill, circle, scale=0.6] (e1) at (2,-1) {};
\node at (3,0) {$\cdots$};

\draw (-2.6,0) -- (a) -- (b) -- (c) -- (d) -- (e) -- (2.6,0);
\draw (a1) -- (a);
\draw (b1) -- (b);
\draw (c1) -- (c);
\draw (d1) -- (d);
\draw (e1) -- (e);

\end{tikzpicture}
\caption{The graph $T_1^\infty$.}\label{comb}
\end{figure}
 
\begin{theorem}\label{thm:comb}
The tree $T_1^\infty$ is a locally finite, $\hat{T}$-free tree with $\zeta(T_1^\infty) = 2$.    
\end{theorem}
\begin{proof} Note that for each $v \in V(T_1^\infty)$, at most two neighbors of $v$ have degree more than one, so $T_1^\infty$ does not contain $\hat{T}$. By Lemma~\ref{lem:finite_ends}, $\zeta(T_1^\infty) \le 2$, so it suffices to prove that one cop cannot capture the robber. We do so by describing a strategy the robber can use to avoid capture by a single cop.

To set up an inductive argument, assume that at the end of some cop turn, the set of contains a vertex on the double ray, say $v$, together with some other vertex. This implies the robber has not been captured, as there must be a unique candidate for the robber to be captured.  Let $u$ and $w$ be the vertices on the double ray on either side of $v$.
Define $u',v,' w'$ as the neighboring vertices of these three that are not on the double ray. 
Let $T_1$ and $T_2$ be the subtrees of $T_1^\infty$ when $v$ and $v'$ are deleted such that $T_1$ contains $u$ and $T_2$ contains $w$.   
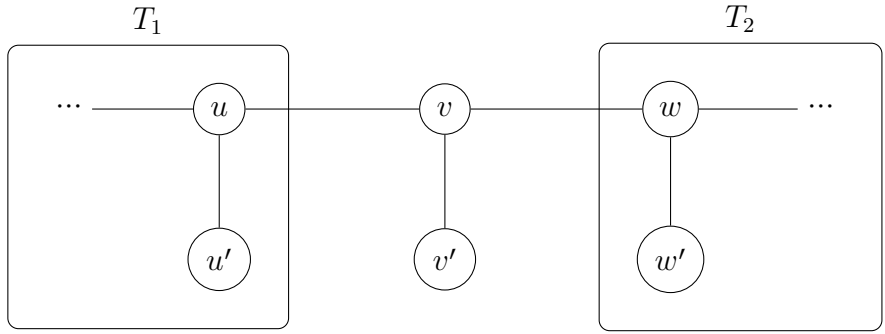
\begin{figure}
    \centering
    \begin{tikzpicture}[>=stealth, auto, node distance=2cm]
    \node[circle, draw] (u) {$u$};
    \node[circle, draw, below of=u] (up) {$u'$};
    \node[circle, draw, right of=u, node distance=3cm] (v) {$v$};
    \node[circle, draw, below of=v] (vp) {$v'$};
    \node[circle, draw, right of=v, node distance=3cm] (w) {$w$};
    \node[circle, draw, below of=w] (wp) {$w'$};
    \node[draw=none] (q_dots_l) [left of =u] {$\cdots$};
    \node[draw=none] (q_dots_r) [right of =w] {$\cdots$};
    
    \draw[-] (u) -- (v);
    \draw[-] (v) -- (w);
    \draw[-] (up) -- (u);
    \draw[-] (vp) -- (v);
    \draw[-] (wp) -- (w);
    \draw[-] (u) -- (q_dots_l); 
    \draw[-] (w) -- (q_dots_r); 
    
    \node[draw, rounded corners, inner sep=0.5cm, fit=(u) (up) (q_dots_l), label={$T_1$}] {};
    \node[draw, rounded corners, inner sep=0.5cm, fit=(w) (wp) (q_dots_r), label={$T_2$}] {};
\end{tikzpicture}
    \caption{The tree $T_1^\infty$ with an arbitrarily chosen vertex $v$ on the double ray.}
    \label{fig:T_infty}
\end{figure}
See Figure~\ref{fig:T_infty} for a visualization of the situation.

The robber takes their move, and the set of candidates now contains $\{u,v,w,v'\}$. 
If the cop plays in $T_1$ or on $v$, then the new set of candidates will contain $w$ and $v'$. 
If the cop plays in $T_2$, then the new set of candidates contains $u$ and $v'$. 
If the cop plays on $v'$, then the new set of candidates contains $u$ and $w$.
The set of candidates includes a vertex on the double ray along with some other vertex. 
Thus, the inductive assumption holds at the end of the cops' next move.
The base case of the induction holds, as the set of candidates before the first cop move is the set of all vertices in the graph. 
By induction, the set of candidates never contains fewer than two elements, and so the robber is never captured. 
\end{proof}

We may show that the tree $T_1^\infty$ is \emph{minimal} in the sense that deleting any edge results in a tree $T$ with $\zeta(T) = 1$. Note that the strategy for avoiding the cop described in Theorem~\ref{thm:comb} relies on the cop being unable to differentiate between a leaf vertex and the next vertex along the double ray. With any edge removed, the cop is able to force the robber towards the anomaly, capturing the robber in the finite part of the tree (if the missing edge disconnects the double ray) or when the missing leaf removes the ambiguity of the robber's position (if the missing edge disconnects a leaf). An interesting problem is determining the minimal locally finite trees with countably (or even finitely) many ends and localization number 2. We think that examples other than $\hat{T}$ and $T_1^\infty$ exist, but it is open whether there exists an infinite family of minimal locally finite trees with two ends and localization number 2.

\section{Subdivisions of graphs}
 Theorem~\ref{thm:delta_lower_bound} implies that the infinite $n(n-1)$-regular tree requires $\Omega(n^4)$ cops, but in Theorem~\ref{thm:achievable} we see subdividing reduced the number of required cops to $n$. This technique was studied in finite graphs, where it is known that every finite graph $G$ has a subdivision $G'$ such that $\zeta(G')=1$; see~\cite{CCDEW12}. We finish with an analogous result holding for locally finite graphs.

\begin{theorem}\label{tsub}
For every locally finite graph $G$, there is a subdivision $G'$ of $G$ such that $\zeta(G') = 1$.
\end{theorem}

\begin{proof}
We first subdivide every edge of $G$ once to obtain $H=G^{1/2}$. Choose $v \in V(H)$ arbitrarily and let $L_i$ be the set of vertices of $V(H)$ at distance exactly $i$ from $v$. Note that $H$ is bipartite, and thus, for every edge $e \in E(H)$, there is $ i \ge 0$ such that $e$ is incident to one vertex in $L_i$ and one vertex in $L_{i+1}$. We subdivide each edge of $H$ between $L_i$ and $L_{i+1}$ by adding
\[ 
    s_i = 2 + |L_i| + |L_{i+1}|
\]
vertices and call the resulting graph $G'$. Note that in $G'$, each vertex of $L_i$ is at distance $D_i = i+ \sum_{j=1}^i s_j$ from $v$. We call vertex $u$ an $H$-\emph{neighbor} of $w$ if $u$ and $w$ are adjacent in $H$.

We claim one cop can capture the robber on $G'$.
The cop starts by probing $v$ and receives distance $d$. If $d=0$, then the robber is captured. Otherwise, there is a unique $i \ge 0$ such that
\[ D_{i-1}< d \le D_i, \]
which implies that the robber is located between $L_{i-1}$ (exclusive) and $L_{i}$ (inclusive). The strategy then proceeds in two phases.

In the first phase, the cop alternates between probing $v$ and a vertex in either $L_i$ or $L_{i+1}$; first probing all vertices of $L_i$, then probing all vertices of $L_{i-1}$ with no vertex probed twice. If the probe at $v$ ever returns some $d$ such that $|d - D_{i-1}| \leq 1$ or $|d - D_i| \leq 1$, then the cop moves on to the second phase. Note that this will happen as soon as the robber visits $S_{i-1}$ or $S_i$; in other words, as long as the cop has not moved on to the second phase, the robber remains in the same path connecting a vertex of $L_{i-1}$ to a vertex of $L_i$. This implies that the robber will be caught since the endpoints of this path are the only vertices where the probe returns a distance $d < s_i$.

Hence, we may assume that the cop moves on to the second phase after finitely many steps. Without loss of generality, assume that a cop probe at $v$ returned $d$ such that $|d - D_i| \leq 1$. The cop then proceeds to probe the vertices in $L_i$, one after the other, until a probe returns $d \leq |L_i|+1$. This is bound to happen because before starting the second phase, the robber was within distance $1$ from some vertex in $L_i$.

Let $w$ be the vertex where a probe returned $d \leq |L_i|+1$. The cop next probes $v$. If this probe returns $D_i$, then the robber is caught at $w$. If it returns $d < D_i$, then the cop proceeds by probing all $H$-neighbors of $w$ in $H_{i-1}$ until one of the probes returns a distance $d \leq s_{i-1}$; at this point, the robber is caught on a path connecting $w$ to the corresponding $H$-neighbor. Note that the robber cannot leave this path through $w$ because this would mean that the next probe returns $d = s_{i-1}$, and they cannot leave it through the leaf in $L_{i-1}$ before being caught because it takes the cop at most $L_{i-1}$ steps to probe all neighbors of $w$ in $L_{i-1}$.

Finally, if the probe at $v$ returns $d > D_i$, then an analogous argument shows that the robber is caught on some path between $w$ and one of its $H$-neighbors in $S_{i+1}$.  \end{proof}

We focused on how results for the Localization game played on finite graphs contrast with the locally finite case. It would be interesting to consider questions only non-trivial in the case of locally finite graphs. One such question is \emph{compactness}: if every finite subgraph of a graph $G$ has localization number at most $M$, can we bound $\zeta(G)$ by some integer-valued function of $M$? Theorem~\ref{thm:achievable} implies that this is not true for locally finite graphs with uncountably many ends as finite subgraphs of $T_\omega$ have localization number at most 2 and yet $\zeta(T_\omega) = \aleph_0$. However, we know of no such examples of locally finite graphs failing compactness with finitely (or even countably) many ends.

\section{Acknowledgments}
The first author was supported by a grant from NSERC.

\end{document}